\documentclass[fleqn]{article}
\usepackage{amsfonts}
\usepackage{amsmath}

\numberwithin{equation}{section}
\newtheorem{theorem}{Theorem}[section]

\newtheorem{definition}[theorem]{Definition}

\newtheorem{lemma}[theorem]{Lemma}

\newtheorem{proposition}[theorem]{Proposition}
\newtheorem{remark}[theorem]{Remark}
\newenvironment{proof}[1][Proof]{\noindent\textbf{#1.} }{\ \rule{0.5em}{0.5em}}

\begin{document}

\author{Abdelhamid AINOUZ \\
Labo. AMNEDP, Dept. of Maths,\\
University of Sciences and Technology, Houari Boumediene,\\
Po Box 32 El Alia Bab Ezzouar, 16111 Algiers, Algeria.\\
ainouz@gmail.com}
\title{Derivation of a Convection Process in a Steady Diffusion-Transfer
Problem by Homogenization\thanks{
\newline
\hspace*{5mm} 2000 Subject Classification: 35B27, 35B40.}}
\maketitle
\date{}

\begin{abstract}
We study the homogenization of a steady diffusion equation in a highly
heterogeneous medium made of two subregions separated by a periodic barrier
through which the flow is proportional to the jump of the temperature by a
layer conductance of the same order of magnitude of the materials in
consideration. The macroscopic governing equations and the effective
conductivity of the homogenized model are obtained by means of the two scale
convergence technique. We show that under some hypothesis the homogenized
systems contain convective terms of order one.
\end{abstract}

\section{Introduction}

Homogenization in multicomponent $\varepsilon -$periodic media with
interfacial barriers has been extensively studied these last years. There
are many mathematical works devoted to the subject and we refer the reader
for instance to Auriault et al\ \cite{ae}, \cite{ar}, Benveniste \cite{ben},
Canon et al. \cite{cp}, Hummel \cite{hummel}, Monsurr\'{o} \cite{mon} and
Pernin \cite{per}\ldots

In \cite{ae}, \cite{ar} the layer conductance or sometimes called the
resistivity is considered as a positive function of order of magnitude $%
\varepsilon ^{\gamma }$ and five distinct macroscopic models are derived by
the formal asymptotic expansion method \cite{blp}, \cite{san}. These
homogenized models are related to five values of $\gamma $ which are $%
-2,-1,0,1$ and $2$. Monsurr\'{o} \cite{mon} considered the same problem and
the derivation of the homogenized models is obtained with the help of the
oscillating test functions method of Tartar \cite{tar}. The case $\gamma
\leq -1$ has been studied by Hummel \cite{hummel} but for media with
disconnected components arranged in a tesselation configuration. He used the
two-scale convergence method by Nguetseng \cite{ngu}, further developed by
Allaire \cite{al} and adapted to periodic surfaces, see for e.g. Allaire et
al \cite{adh} and Neuss-Radu \cite{nr}. The aim of this paper is to consider
in the simplest configuration the case of resistivity having zero average
value on the periodic interface. As in Ainouz \cite{ain} we show that the
homogenized problem contains convective terms.

The paper is organized as follows: in Section 2 we give the geometry of the
medium in which the stationary diffusion micro-model is set. In Section 3
the weak formulation of the problem is studied in a suitable function space
and we give the main a priori estimates. Finally in Section 4, we recall the
two-scale convergence and its main results, which we shall use it to derive
the homogenized problem with the help of two-scale convergence method.

\section{Setting of the Problem and statement of the main result}

We begin this section by describing the geometry of the medium underlying
the exact micro-model for the steady diffusion equation in a two-component
periodic medium.

Let $Y=(0,1)^{n}$ be the unit cell of periodicity and assume that $Y$ is
divided as $Y=Y_{1}\cup Y_{2}\cup \Sigma $ where $Y_{i}$, $i=1,2$ are two
open subsets of $Y$ and $\Sigma =\partial Y_{1}\cap \partial Y_{2}$ is a
sufficiently smooth interface that separates them. The sets $Y_{1}$ and $%
Y_{2}$ are made of two different materials but having conductance of same
order of magnitude and $\Sigma $ is a thin layer of very low conductance
which constitutes a flow exchange barrier.

Let $A_{i}(y)$ ($i=1,2$) denotes the conductivity tensor of the material $%
Y_{i}$. We assume that $A_{i}(y)$ is a $n\times n$ $Y$-periodic
matrix-valued function and continuous on $\mathbb{R}^{n}$ such that
\begin{equation}
m_{i}|\lambda |^{2}\leq A_{i}(y)\lambda \cdot \lambda  \label{1}
\end{equation}%
and
\begin{equation}
A_{i}(y)\lambda \cdot \eta \leq M_{i}|\lambda ||\eta |  \label{2}
\end{equation}%
for a.e. $y\in Y_{i}$ and for all $\lambda ,\eta \in \mathbb{R}^{n},$ where $%
m_{i},\ M_{i}$ are positive constants.

Let $\alpha $ denotes the barrier resistivity of $\Sigma $. \ For simplicity
we suppose that $\alpha $ is a continuous function on $\mathbb{R}^{n}$, $Y-$%
periodic with a zero average-value on $\Sigma $, i.e.
\begin{equation}
\int_{\Sigma }\alpha (y)d\sigma (y)=0.  \label{3}
\end{equation}%
Here $d\sigma (y)$ denotes the surface measure on $\Sigma $. We decompose $%
\alpha $ into its positive and negative parts as follows:%
\begin{equation}
\alpha =\alpha ^{+}-\alpha ^{-},  \label{4}
\end{equation}%
where
\begin{equation*}
\alpha ^{+}=\sup \left( \alpha ,0\right) ,\quad \alpha ^{-}=\sup \left(
-\alpha ,0\right) .
\end{equation*}%
We assume that%
\begin{equation}
\alpha ^{+}\geq \alpha _{0}  \label{5}
\end{equation}%
where $\alpha _{0}$ is a positive real number.

Let $a_{i}\left( y\right) $ ($i=1,2$) be continuous functions on $\mathbb{R}%
^{n}$ and $Y-$periodic such that
\begin{equation}
a_{i}\left( y\right) \geq \eta _{i}>0\text{ for a.e. }y\in Y_{i}.  \label{7}
\end{equation}

Let $\Omega $ be a bounded domain in $\mathbb{R}^{n}$ with sufficiently
smooth boundary $\Gamma $. Let $\varepsilon >0$ be a positive number taking
its value in a sequence of real numbers converging to zero. We consider the
following open subsets of $\Omega $%
\begin{equation}
\Omega _{i}^{\varepsilon }=\{x\in \Omega ;\ \chi _{i}\left( \frac{x}{%
\varepsilon }\right) =1\}\quad i=1,2  \label{8}
\end{equation}%
where $\chi _{i}\left( y\right) $ is the $Y-$periodic characteristic
function of $Y_{i}$.

The $(n-1)$-dimensional surface $\Sigma ^{\varepsilon }$ defined by
\begin{equation}
\Sigma ^{\varepsilon }=\{x\in \Omega ;\ \exists \overrightarrow{k}\in
\mathbb{Z}^{n},\ \frac{x}{\varepsilon }+\overrightarrow{k}\in \Sigma \}
\label{9}
\end{equation}%
is then by construction the interfacial barrier that separates the materials
$\Omega _{1}^{\varepsilon }$ and $\Omega _{2}^{\varepsilon }$, i.e. $\Sigma
^{\varepsilon }=\partial \Omega _{1}^{\varepsilon }\cap \partial \Omega
_{2}^{\varepsilon }$.

Let $\Gamma _{i}^{\varepsilon }=\overline{\Omega _{i}^{\varepsilon }}\cap
\Gamma $ ($i=1,2$) and assume that $|\Gamma _{i}^{\varepsilon }|\neq 0$. Let
$f_{i}^{\varepsilon }\in L^{2}(\Omega _{i}^{\varepsilon })$ ($i=1,2$) be
uniformly bounded functions, that is,
\begin{equation}
||f_{i}^{\varepsilon }||_{0,\Omega _{i}^{\varepsilon }}\leq C  \label{10}
\end{equation}%
where (and throughout this paper) $C$ is a positive constant independent of $%
\varepsilon $.

Let us consider the following transmission problem:%
\begin{equation}
-\text{\textrm{\text{div}}}(A_{i}(\frac{x}{\varepsilon })\nabla
u_{i}^{\varepsilon }(x))+a_{i}(\frac{x}{\varepsilon })u_{i}^{\varepsilon
}(x)=f_{i}^{\varepsilon }\left( x\right) \text{ in }\Omega _{i}^{\varepsilon
},\text{\qquad }i=1,2  \label{11}
\end{equation}%
\begin{equation}
A_{1}(\frac{x}{\varepsilon })\nabla u_{1}^{\varepsilon }(x)\cdot \nu
^{\varepsilon }=A_{2}(\frac{x}{\varepsilon })\nabla u_{2}^{\varepsilon
}(x)\cdot \nu ^{\varepsilon }\text{ on }\Sigma ^{\varepsilon },  \label{12}
\end{equation}

\begin{equation}
A_{1}(\frac{x}{\varepsilon })\nabla u_{1}^{\varepsilon }(x)\cdot \nu
^{\varepsilon }=-\alpha (\frac{x}{\varepsilon })\left( u_{1}^{\varepsilon
}-u_{2}^{\varepsilon }\right) \text{ on }\Sigma ^{\varepsilon },  \label{13}
\end{equation}%
\begin{equation}
u_{i}^{\varepsilon }=0\text{ on }\Gamma _{i}^{\varepsilon }\text{.\qquad }%
i=1,2  \label{14}
\end{equation}%
where $\nu ^{\varepsilon }$ is the unit outward normal to $\Omega
_{1}^{\varepsilon }$ obtained in an obvious way by extending $\nu $, the
unit outward normal to $Y_{1}$, by $Y-$periodicity to the whole space $%
\mathbb{R}^{n}$.

The first equation (\ref{11}) is the classical steady diffusion problem. The
second equation (\ref{12}) expresses the continuity of the flow across the
interfacial contact $\Sigma ^{\varepsilon }$ while the boundary condition (%
\ref{13}) says that the flow is proportional to the jump of the temperature.
For a physical justification of the condition (\ref{13}) we refer the reader
to Carslaw and Jaeger \cite{cj}, Kholodovskii \cite{kh}. Finally the last
condition (\ref{14}) is the well-known Dirichlet boundary condition.

In order to state the main theorem of this paper we begin by setting some
notations. Let us first consider the following microscopic problems:%
\begin{equation}
-\mathrm{div}_{y}(A_{i}(y)(\nabla _{y}\gamma _{i}(y))=0\text{ in }Y_{i},%
\text{\qquad }  \label{15}
\end{equation}%
\begin{equation}
(A_{i}(y)(\nabla _{y}\gamma _{i}(y))\cdot \nu (y)=-\alpha (y)\text{ on }%
\Sigma ,  \label{16}
\end{equation}%
\textbf{\ }%
\begin{equation}
y\longmapsto \gamma _{i}(y)\text{ $Y$-periodic,\qquad }i=1,2,  \label{17}
\end{equation}%
and for $k=1,2,\ldots ,n$ the problem
\begin{equation}
-\mathrm{div}_{y}(A_{i}(y)(e^{k}+\nabla _{y}\xi _{i}^{k}(y))=0\text{ in }%
Y_{i},  \label{18}
\end{equation}%
\begin{equation}
(A_{i}(y)(e^{k}+\nabla _{y}\xi _{i}^{k}(y))\cdot \nu (y)=0\text{ on }\Sigma ,
\label{19}
\end{equation}%
\begin{equation}
y\longmapsto \xi _{i}^{k}(y)\text{ $Y$-periodic,\qquad }i=1,2,\text{.}
\label{20}
\end{equation}%
where $\mathrm{div}_{y}$ (resp.$\nabla _{y}$ ) is the divergence (resp.
gradient) operator with respect to the variable $y$ and $(e^{k})_{1\leq
k\leq n}$\ is the canonical basis of $\mathbb{R}^{n}$.

The well-posedness of these elliptic problems is well-known, see for
instance Bensoussan et al. \cite{blp}. Let us mention that the assumption (%
\ref{3}) is a compatibility condition for the unique solvability of the
problems (\ref{15})-(\ref{17}).

Let us define the effective matrices $A_{i}^{eff}$\ $=(a_{i}^{kj})_{1\leq
k,j\leq n}$ and the convection vectors $B_{i}=(b_{i}^{k})_{1\leq k\leq n}$ ($%
i=1,2$)%
\begin{equation}
a_{i}^{kj}=\int_{Y_{i}}A_{i}(e^{k}+\nabla _{y}\xi _{i}^{k})\cdot
(e^{j}+\nabla _{y}\xi _{i}^{j})dy,\quad  \label{21}
\end{equation}%
\begin{equation}
b_{i}^{k}=(-1)^{i-1}\int_{Y_{i}}A_{i}\nabla \gamma _{i}dy+\int_{\Sigma
}\alpha \xi _{i}^{k}d\sigma (y).  \label{22}
\end{equation}%
Let us denote
\begin{equation}
d=\int_{\Sigma }\alpha (\gamma _{1}-\gamma _{2})d\sigma (y)  \label{23}
\end{equation}%
\textbf{\ }and%
\begin{equation}
c_{i}=d+\int_{Y_{i}}a_{i}dy,\quad g_{i}(x)=\int_{Y_{i}}f_{i}(x,y)dy,\qquad
i=1,2  \label{24}
\end{equation}%
where $f_{i}$ is the two-scale limit of the sequence $(\chi _{i}(\frac{x}{%
\varepsilon })f_{i}^{\varepsilon })_{\varepsilon >0}$.

The main result of the paper is the following theorem

\begin{theorem}
\label{mr}Let $(u_{1}^{\varepsilon },u_{2}^{\varepsilon })$\ be the solution
of the problem (\ref{11})-(\ref{14}). Then up to a subsequence, $(\chi _{1}(%
\frac{x}{\varepsilon })u_{1}^{\varepsilon },\chi _{2}(\frac{x}{\varepsilon }%
)u_{2}^{\varepsilon })_{\varepsilon >0}$\ two-scale converges to $\left(
u_{1},u_{2}\right) \in (H_{0}^{1}(\Omega ))^{2}$\ solution of the
Homogenized Problem:%
\begin{equation}
-\mathrm{div}(A_{1}^{eff}\nabla u_{1})+B_{1}\cdot \nabla u_{1}-B_{2}\cdot
\nabla u_{2}+c_{1}u_{1}-du_{2}=g_{1}\text{ in }\Omega \text{,}  \label{25}
\end{equation}%
\begin{equation}
-\mathrm{div}(A_{2}^{eff}\nabla u_{2})+B_{2}\cdot \nabla u_{2}-B_{1}\cdot
\nabla u_{1}+c_{2}u_{2}-du_{1}=g_{2}\text{ in }\Omega \text{,}  \label{26}
\end{equation}%
\begin{equation}
u_{1}=u_{2}=0\text{ on }\Gamma .  \label{27}
\end{equation}
\end{theorem}

\section{Solvability of the problem and a priori estimates}

Let us introduce the Hilbert space $V^{\varepsilon }=H^{1}\left( \Omega
_{1}^{\varepsilon },\Gamma _{1}^{\varepsilon }\right) \times H^{1}\left(
\Omega _{2}^{\varepsilon },\Gamma _{2}^{\varepsilon }\right) $ where
\begin{equation}
H^{1}\left( \Omega _{i}^{\varepsilon },\Gamma _{i}^{\varepsilon }\right)
=\{v_{i}\in H^{1}\left( \Omega _{i}^{\varepsilon }\right) ;v_{i}=0\text{ on }%
\Gamma _{i}^{\varepsilon }\}.\qquad i=1,2  \label{28}
\end{equation}%
$V^{\varepsilon }$ is equipped with the norm:
\begin{equation}
||(v_{1},v_{2})||_{V^{\varepsilon }}^{2}:=||v_{1}||_{1,\Omega
_{1}^{\varepsilon }}^{2}+||v_{2}||_{1,\Omega _{2}^{\varepsilon
}}^{2}+||v_{1}-v_{2}||_{0,\Sigma ^{\varepsilon }}^{2}\text{.}  \label{29}
\end{equation}

The equivalent weak formulation of the micro-model (\ref{11})-(\ref{14}) is :%
\begin{equation}
\begin{array}{l}
\text{For each }\varepsilon >0\text{, find }(u_{1}^{\varepsilon
},u_{2}^{\varepsilon })\in V^{\varepsilon }\text{ such that} \\
a^{\varepsilon }((u_{1}^{\varepsilon },u_{2}^{\varepsilon
}),(v_{1},v_{2}))=L^{\varepsilon }((v_{1},v_{2}))\text{ for all }%
(v_{1},v_{2})\in V^{\varepsilon }%
\end{array}
\label{30}
\end{equation}%
where for all $(w_{1},w_{2}),(v_{1},v_{2})\in V^{\varepsilon }$%
\begin{eqnarray}
a^{\varepsilon }((w_{1},w_{2}),(v_{1},v_{2})) &=&\int_{\Omega }(A(\frac{x}{%
\varepsilon })\nabla w^{\varepsilon }\cdot \nabla v^{\varepsilon }+a(\frac{x%
}{\varepsilon })w^{\varepsilon }v^{\varepsilon })dx+  \notag \\
&&+\int_{\Sigma ^{\varepsilon }}\alpha (\frac{x}{\varepsilon }%
)(w_{1}-w_{2})(v_{1}-v_{2})d\sigma ^{\varepsilon }(x),  \label{31}
\end{eqnarray}%
\begin{equation}
L^{\varepsilon }((v_{1},v_{2}))=\left( \int_{\Omega }f^{\varepsilon
}(x)v^{\varepsilon }(x)dx\right) .\   \label{32}
\end{equation}%
Here $d\sigma ^{\varepsilon }(x)$ denotes the surface measure on $\Sigma
^{\varepsilon }$ and $A,w^{\varepsilon },v^{\varepsilon }$, $f^{\varepsilon
} $ are defined by%
\begin{equation*}
A(y)=\chi _{1}(y)A_{1}(y)+\chi _{2}(y)A_{2}(y),\ y\in \mathbb{R}^{n}\text{,}
\end{equation*}%
\begin{eqnarray*}
w^{\varepsilon } &=&\chi _{1}(\frac{x}{\varepsilon })w_{1}(x)+\chi _{2}(%
\frac{x}{\varepsilon })w_{2}(x),\ x\in \Omega , \\
v^{\varepsilon } &=&\chi _{1}(\frac{x}{\varepsilon })v_{1}(x)+\chi _{2}(%
\frac{x}{\varepsilon })v_{2}(x),\ x\in \Omega , \\
f^{\varepsilon } &=&\chi _{1}(\frac{x}{\varepsilon })f_{1}^{\varepsilon
}(x)+\chi _{2}(\frac{x}{\varepsilon })f_{2}^{\varepsilon }(x),x\in \Omega .
\end{eqnarray*}%
To study the solvability of the problem (\ref{30}) we shall use the
following inequality:

\begin{lemma}
\label{l1}There exists a constant $C_{i}>0$ independent of $\varepsilon $
such that for every $v_{i}\in H^{1}\left( \Omega _{i}^{\varepsilon }\right) $
and for all $\delta _{i}>0$ we have
\begin{equation}
||v_{i}||_{0,\Sigma ^{\varepsilon }}^{2}\leq C_{i}(\frac{1}{\delta
_{i}\varepsilon }||v_{i}||_{0,\Omega _{i}^{\varepsilon }}^{2}+\delta
_{i}\varepsilon ||\nabla v_{i}||_{0,\Omega _{i}^{\varepsilon }}^{2})\text{%
.\qquad }i=1,2  \label{33}
\end{equation}
\end{lemma}

\begin{proof}
See for instance Ainouz \cite{ain} (see also Monsurr\'{o} \cite{mon}). Note
that the constants $C_{i}$ depends solely on the geometry of $Y_{i}$.\hfill
\end{proof}

Throughout this paper we assume that
\begin{equation}
||\alpha ^{-}||_{\infty ,\Sigma ^{\varepsilon }}\leq \ \underset{i=1,2}{\min
}\frac{\sqrt{m_{i}\eta _{i}}}{C_{i}}.  \label{34}
\end{equation}%
where $C_{i}$ are the constants defined in Lemma \ref{l1}. Now we give the
following existence result :

\begin{proposition}
Let the assumption (\ref{34}) be fulfilled. Then for any fixed $\varepsilon
>0$, there exists a unique couple $(u_{1}^{\varepsilon },u_{2}^{\varepsilon
})\in V^{\varepsilon }$ solution of (\ref{30}). Moreover we have the a
priori estimate
\begin{equation}
||(u_{1}^{\varepsilon },u_{2}^{\varepsilon })||_{V^{\varepsilon }}\leq C.
\label{35}
\end{equation}
\end{proposition}

\begin{proof}
First we show that $a^{\varepsilon }(\cdot ,\cdot )$ is coercive on $%
V^{\varepsilon }$. Let $(v_{1},v_{2})\in V^{\varepsilon }$, then
\begin{equation}
\left.
\begin{array}{c}
a^{\varepsilon }((v_{1},v_{2}),(v_{1},v_{2}))\geq \sum_{i=1,2}(m_{i}||\nabla
v_{i}||_{0,\Omega _{i}^{\varepsilon }}^{2}+\eta _{i}||v_{i}||_{0,\Omega
_{i}^{\varepsilon }}^{2}) \\
+(\alpha _{0}-||\alpha ^{-}||_{\infty ,\Sigma ^{\varepsilon }})\int_{\Sigma
^{\varepsilon }}\left( v_{1}(x)-v_{2}(x)\right) ^{2}d\sigma ^{\varepsilon
}(x).%
\end{array}%
\right.  \label{36}
\end{equation}%
But in view of Lemma \ref{l1} we see that
\begin{equation*}
\int_{\Sigma ^{\varepsilon }}\left( v_{1}(x)-v_{2}(x)\right) ^{2}d\sigma
^{\varepsilon }(x)\leq \sum_{i=1,2}C_{i}(\frac{1}{\delta _{i}\varepsilon }%
||v_{i}||_{0,\Omega _{i}^{\varepsilon }}^{2}+\delta _{i}\varepsilon ||\nabla
v_{i}||_{0,\Omega _{i}^{\varepsilon }}^{2}).
\end{equation*}%
Set
\begin{equation*}
\beta _{i}=m_{i}-C_{i}\delta _{i}\varepsilon ||\alpha ^{-}||_{\infty ,\Sigma
^{\varepsilon }},\ \gamma _{i}:=\eta _{i}-\frac{C_{i}}{\delta
_{i}\varepsilon }||\alpha ^{-}||_{\infty ,\Sigma ^{\varepsilon }}\text{.}
\end{equation*}%
Therefore (\ref{36}) becomes
\begin{eqnarray}
a^{\varepsilon }((v_{1},v_{2}),(v_{1},v_{2})) &\geq &\sum_{i=1,2}\left(
\beta _{i}||\nabla v_{i}||_{0,\Omega _{i}^{\varepsilon }}^{2}++\gamma
_{i}||v_{i}||_{0,\Omega _{i}^{\varepsilon }}^{2}\right)  \notag \\
&&+\alpha _{0}||v_{1}-v_{2}||_{0,\Sigma ^{\varepsilon }}^{2}.  \label{37}
\end{eqnarray}%
Now we choose, for example, $\delta _{i}=\dfrac{2\varepsilon
^{-1}C_{i}m_{i}||\alpha ^{-}||_{\infty ,\Sigma ^{\varepsilon }}}{\eta
_{i}m_{i}+(C_{i}||\alpha ^{-}||_{\infty ,\Sigma ^{\varepsilon }})^{2}},\quad
i=1,2$. Then by (\ref{34})\ we have%
\begin{equation}
\beta _{i}=\frac{m_{i}(\eta _{i}m_{i}-(C_{i}||\alpha ^{-}||_{\infty ,\Sigma
^{\varepsilon }})^{2})}{\eta _{i}m_{i}+(C_{i}||\alpha ^{-}||_{\infty ,\Sigma
^{\varepsilon }})^{2}}>0,  \label{38}
\end{equation}%
and%
\begin{equation}
\gamma _{i}=\frac{\eta _{i}m_{i}-(C_{i}||\alpha ^{-}||_{\infty ,\Sigma
^{\varepsilon }})^{2}}{2C_{i}m_{i}||\alpha ^{-}||_{\infty ,\Sigma
^{\varepsilon }}}>0.  \label{39}
\end{equation}%
Hence by (\ref{38}) and (\ref{39}), the inequality (\ref{37}) becomes
\begin{equation}
a^{\varepsilon }((v_{1},v_{2}),(v_{1},v_{2}))\geq
c_{0}||(v_{1},v_{2})||_{V^{\varepsilon }}^{2}  \label{40}
\end{equation}%
where
\begin{equation*}
c_{0}=\min (\beta _{1},\beta _{2},\gamma _{1},\gamma _{2},\alpha _{0})>0
\end{equation*}%
which is independent of $\varepsilon $.

It is easy to observe that $a^{\varepsilon }((w_{1},w_{2}),(v_{1},v_{2}))$
is bilinear continuous on $V^{\varepsilon }\times V^{\varepsilon }$ and that
$L^{\varepsilon }((v_{1},v_{2}))$ is linear and continuous on $%
V^{\varepsilon }$. Consequently by the Lax-Milgram Lemma the problem (\ref%
{30}) has a unique solution $(u_{1}^{\varepsilon },u_{2}^{\varepsilon })$ in
$V^{\varepsilon }$. Furthermore we have%
\begin{equation*}
||(u_{1}^{\varepsilon },u_{2}^{\varepsilon })||_{V^{\varepsilon }}^{2}\leq
\frac{1}{c_{0}}L^{\varepsilon }((u_{1}^{\varepsilon },u_{2}^{\varepsilon
}))\leq \frac{1}{c_{0}}\underset{i=1,2}{\sum }||f_{i}^{\varepsilon
}||_{0,\Omega _{i}^{\varepsilon }}\ ||u_{i}^{\varepsilon }||_{0,\Omega
_{i}^{\varepsilon }}
\end{equation*}%
which implies by (\ref{10}) that the sequence $(u_{1}^{\varepsilon
},u_{2}^{\varepsilon })_{\varepsilon >0}$ is uniformly bounded in $%
V^{\varepsilon }$. The Proposition is then proved.\hfill
\end{proof}

In view of the a priori estimate (\ref{35}), one is interested in
investigating the limit of the sequence $((u_{1}^{\varepsilon
},u_{2}^{\varepsilon }))_{\varepsilon >0}$ as $\varepsilon \rightarrow 0$ in
a sense that will be specified later. This is the purpose of the next
Section.

\section{two-scale convergence process}

In this Section we shall use the two-scale convergence method to determine
the Homogenization of the Problem (\ref{11})-(\ref{14}). For more details on
two-scale convergence, we refer the reader to Nguesteng \cite{ngu}, Allaire
\cite{al} and a recent paper by Lukkassen et al. \cite{lnw}. We also mention
the work by Allaire et al. \cite{adh} (see also Ainouz \cite{ainthesis},
Neuss-radu \cite{nr}) where the two-scale convergence is applied to periodic
surfaces. For convenience we recall the definition and the main compactness
results of this method.

\begin{definition}
\label{d1}\qquad

\begin{enumerate}
\item \label{d11}Let $(v^{\varepsilon })_{\varepsilon >0}$ be a sequence in $%
L^{2}(\Omega )$ and $v_{0}\in L^{2}(\Omega \times Y)$. We say that $%
(v^{\varepsilon })_{\varepsilon >0}$ two-scale converges\ to $v_{0}$ if for
every\textbf{\ }$\varphi \in \mathcal{D(}\Omega ;\mathcal{C}_{per}^{\infty
}\left( Y\right) )$ we have%
\begin{equation}
\underset{\varepsilon \rightarrow 0}{\lim }\int_{\Omega }v^{\varepsilon
}(x)\varphi (x,\frac{x}{\varepsilon })dx=\int_{\Omega \times
Y}v_{0}(x,y)\varphi (x,y)dxdy\text{.}  \label{41}
\end{equation}

\item \label{d12}Let $(v^{\varepsilon })_{\varepsilon >0}$ be a sequence in $%
L^{2}(\Sigma ^{\varepsilon })$ and $v_{0}\in L^{2}(\Omega \times \Sigma )$.
We say that $(v^{\varepsilon })_{\varepsilon >0}$ two-scale converges\ to $%
v_{0}$ if for every\textbf{\ }$\varphi \in \mathcal{D(}\overline{\Omega };%
\mathcal{C}_{per}^{\infty }\left( Y\right) )$ we have%
\begin{equation}
\underset{\varepsilon \rightarrow 0}{\lim }\varepsilon \int_{\Sigma
^{\varepsilon }}v^{\varepsilon }(x)\varphi (x,\frac{x}{\varepsilon })d\sigma
^{\varepsilon }(x)=\int_{\Omega \times \Sigma }v_{0}(x,y)\varphi
(x,y)dxd\sigma (y)\text{.}  \label{41p}
\end{equation}
\end{enumerate}
\end{definition}

We now give the main compactness results of the two scale convergence :

\begin{theorem}
\label{l2}\

\begin{enumerate}
\item Let $(v^{\varepsilon })_{\varepsilon >0}$ be a uniformly bounded
sequence in $L^{2}(\Omega )$. Then one can extract a subsequence which
two-scale converges to a function $v_{0}\in L^{2}(\Omega \times Y)$ in the
sense of Definition \ref{d1}\ref{d11}.

\item Let $(\sqrt{\varepsilon }v^{\varepsilon })_{\varepsilon >0}$ be a
uniformly bounded sequence in $L^{2}(\Sigma ^{\varepsilon })$. Then one can
extract a subsequence which two-scale converges to a function $w_{0}\in $ $%
L^{2}(\Omega \times \Sigma )$.

\item If $(v^{\varepsilon })_{\varepsilon >0}$ is uniformly bounded in $%
H^{1}(\Omega )$ (resp. $H_{0}^{1}(\Omega )$) then one can extract a
subsequence still denoted $(v^{\varepsilon })_{\varepsilon >0}$ and there
exist $v\in H^{1}(\Omega )$ (resp. $H_{0}^{1}(\Omega )$) and $V\in
L^{2}(\Omega ;H_{per}^{1}(Y)/\mathbb{R})$ such that

\begin{description}
\item[i)] $(v^{\varepsilon })_{\varepsilon >0}$ two-scale converges to $v$
in the sense of Definition \ref{d1}\ref{d11}.

\item[ii)] $(\nabla v^{\varepsilon })_{\varepsilon >0}$ two-scale converges
to $\nabla v+\nabla _{y}V$ in the sense of Definition \ref{d1}\ref{d11}.

\item[iii)] Moreover $(\sqrt{\varepsilon }v^{\varepsilon })_{\varepsilon >0}$
is uniformly bounded in $L^{2}(\Sigma ^{\varepsilon })$ and $(v^{\varepsilon
})_{\varepsilon >0}$ two-scale converges to $v$ in the sense of Definition %
\ref{d1}\ref{d12}. we have $w_{0}=v_{0}$.
\end{description}
\end{enumerate}
\end{theorem}

Next we shall apply these results to determine the two-scale limit of $%
(u_{i}^{\varepsilon })_{\varepsilon >0}$ and the source terms $\left(
f_{i}^{\varepsilon }\right) _{\varepsilon >0}$, $i=1,2$.

\begin{lemma}
\label{l3}For each $i=1,2$ there exist $f_{i}\in L^{2}(\Omega \times Y_{i})$%
, $u_{i}(x)\in H_{0}^{1}(\Omega )$ and $U_{i}\in L^{2}(\Omega
,H_{per}^{1}(Y_{i})/\mathbb{R})$ such that up to a subsequence we have for
all $\varphi _{i}\in \mathcal{D(}\Omega ;\mathcal{C}_{per}^{\infty }\left(
Y_{i}\right) )$ and $\psi _{i}\in \mathcal{D(}\Omega ;\mathcal{C}%
_{per}^{\infty }\left( Y_{i}\right) )^{n}$:%
\begin{equation}
\underset{\varepsilon \rightarrow 0}{\lim }\int_{\Omega _{i}^{\varepsilon
}}f_{i}^{\varepsilon }(x)\varphi _{i}(x,\frac{x}{\varepsilon }%
)dx=\int_{\Omega }\int_{Y_{i}}f_{i}(x,y)\varphi _{i}(x,y)dydx,  \label{42}
\end{equation}

\begin{equation}
\underset{\varepsilon \rightarrow 0}{\lim }\int_{\Omega _{i}^{\varepsilon
}}u_{i}^{\varepsilon }(x)\varphi _{i}(x,\frac{x}{\varepsilon }%
)dx=\int_{\Omega }\int_{Y_{i}}u_{i}(x)\varphi _{i}(x,y)dydx,  \label{43}
\end{equation}

\begin{equation}
\underset{\varepsilon \rightarrow 0}{\lim }\int_{\Omega }\nabla
u_{i}^{\varepsilon }(x)\psi _{i}(x,\frac{x}{\varepsilon })dx=\int_{\Omega
}\int_{Y_{i}}[\nabla u_{i}(x)+\nabla _{y}U_{i}(x,y)]\psi _{i}(x,y)dydx
\label{44}
\end{equation}%
and for all $\varphi \in \mathcal{D(}\overline{\Omega };\mathcal{C}%
_{per}^{\infty }\left( Y\right) )$%
\begin{equation}
\underset{\varepsilon \rightarrow 0}{\lim }\int_{\Sigma ^{\varepsilon
}}\varepsilon u_{i}^{\varepsilon }(x)\varphi (x,\frac{x}{\varepsilon }%
)d\sigma ^{\varepsilon }(x)=\int_{\Omega }\int_{\Sigma }u_{i}(x)\varphi
(x,y)d\sigma (y)dx,\qquad i=1,2.  \label{45}
\end{equation}
\end{lemma}

\begin{proof}
The existence of the two-scale limits is an immediate consequence of the a
priori estimates (\ref{10}), (\ref{35}) and the definition of two-scale
convergence. Indeed, since for $i=1,2$ the sequences $\chi _{i}(\frac{x}{%
\varepsilon })f_{i}^{\varepsilon }(x),$ $\chi _{i}(\frac{x}{\varepsilon }%
)u_{i}^{\varepsilon }(x)_{\varepsilon >0},$ $(\chi _{i}(\frac{x}{\varepsilon
})\nabla u_{i}^{\varepsilon }(x))_{\varepsilon >0}$\ and $(\sqrt{\varepsilon
}u_{i}^{\varepsilon }(x))_{\varepsilon >0}$ are uniformly bounded in $%
L^{2}(\Omega )$,$\ L^{2}(\Omega )$, $(L^{2}(\Omega ))^{n}$ and $L^{2}(\Sigma
^{\varepsilon })$ respectively. Then by Theorem \ref{l2}, one can extract a
subsequence still denoted $\varepsilon $ and there exist $f_{i}(x,y),\
u_{i}^{0}(x,y)\in L^{2}(\Omega \times Y)$ and $\xi _{i}\in (L^{2}(\Omega
\times Y))^{n}$ such that $\chi _{i}(\frac{x}{\varepsilon }%
)f_{i}^{\varepsilon }(x)$, $\chi _{i}(\frac{x}{\varepsilon }%
)u_{i}^{\varepsilon }(x)$ and $\chi _{i}(\frac{x}{\varepsilon })\nabla
u_{i}^{\varepsilon }(x)$ two-scale converge respectively to $f_{i}(x,y)$, $%
u_{i}(x,y)$ and $\xi _{i}(x,y)$. We point out that $f_{i}(x,y),\ u_{i}(x,y)$
and $\xi _{i}(x,y)$ are equal zero outside $Y_{i}$. Arguing as in Allaire
\cite[Thm 2.9]{al} we easilty arrive at $u_{i}^{0}(x,y)=\chi _{i}(y)u_{i}(x)$
where $u_{i}(x)\in H_{0}^{1}(\Omega )$ and $\xi _{i}(x,y)=\nabla
u_{i}(x)+\nabla _{y}U_{i}(x,y)$ where $U_{i}\in L^{2}(\Omega
,H_{per}^{1}(Y_{i})/\mathbb{R})$. Finally, by Allaire et al. \cite[Prop. 2.6]%
{adh} (see also Ainouz \cite{ainthesis} or Neuss-Radu \cite[Thm. 2.2]{nr}))
we obtain the last limit.
\end{proof}

The following result is an extension of an auxiliary result given in \cite%
{ain}.

\begin{lemma}
\label{lem2}Up to the subsequence given in Lemma \ref{l3}, we have
\begin{equation*}
\underset{\varepsilon \rightarrow 0}{\lim }\int_{\Sigma ^{\varepsilon
}}u_{i}^{\varepsilon }(x)\varphi _{i}(x,\frac{x}{\varepsilon })d\sigma
^{\varepsilon }(x)=\int_{\Omega }\int_{\Sigma }U_{i}(x,y)\varphi
_{i}(x,y)d\sigma (y)dx
\end{equation*}%
for any $\varphi _{i}\in \mathcal{D(}\overline{\Omega };\mathcal{C}%
_{per}\left( \overline{Y_{i}}\right) )$ such that $\int_{\Sigma }\varphi
_{i}(x,y)d\sigma (y)=0\qquad i=1,2.$
\end{lemma}

\begin{proof}
For a fixed $x$ in $\overline{\Omega }$, and $i=1,2$ let $\phi _{i}(x,y)$ be
the solution of the following boundary value problem
\begin{equation}
-\mathrm{div}_{y}\phi _{i}(x,y)=0\text{ in }Y_{i}  \label{46}
\end{equation}%
\begin{equation}
\phi _{i}(x,y)\cdot \nu (y)=\left( -1\right) ^{i-1}\varphi _{i}(x,y)\text{
on }\Sigma ,  \label{47}
\end{equation}%
\begin{equation}
y\longmapsto \phi _{i}(x,y)\text{ is }Y\text{-periodic.}  \label{48}
\end{equation}%
Such a function exists since $\int_{\Sigma }\varphi _{i}(x,y)d\sigma (y)=0$.
Furthermore, the solution $\phi _{i}(x,\cdot )$ belongs to $%
H_{per}^{1}(Y_{i})/\mathbb{R}$. Taking into account the boundary conditions (%
\ref{47}) and (\ref{48}) we see that
\begin{equation}
\int_{\Sigma ^{\varepsilon }}u_{i}^{\varepsilon }(x)\varphi _{i}(x,y)d\sigma
^{\varepsilon }(x)=\int_{\Omega _{i}^{\varepsilon }}\nabla
u_{i}^{\varepsilon }(x)\phi _{i}(x,\frac{x}{\varepsilon })dx+\int_{\Omega
_{i}^{\varepsilon }}u_{i}^{\varepsilon }(x)\mathrm{div}_{x}(\phi _{i}(x,%
\frac{x}{\varepsilon }))dx.  \label{49}
\end{equation}%
By (\ref{46}) and letting $\varepsilon \rightarrow 0$ in (\ref{49}) together
with (\ref{43}), (\ref{44}) we have
\begin{eqnarray}
\underset{\varepsilon \rightarrow 0}{\lim }\int_{\Sigma ^{\varepsilon
}}u_{i}^{\varepsilon }(x)\varphi _{i}(x,y)d\sigma ^{\varepsilon }(x)
&=&\int_{\Omega }\int_{Y_{i}}[\nabla u_{i}(x)+\nabla _{y}U_{i}(x,y)]\phi
_{i}(x,y)dydx  \notag \\
&&+\int_{\Omega }\int_{Y_{i}}u_{i}(x)\mathrm{div}_{x}\phi _{i}(x,y)dydx\text{%
.}  \label{50}
\end{eqnarray}%
Since $u_{i}\in H_{0}^{1}(\Omega )$, integration by parts with respect to $x$
gives
\begin{equation*}
\int_{\Omega }\nabla u_{i}(x)\phi _{i}(x,y)dx+\int_{\Omega }u_{i}(x)\mathrm{%
div}_{x}\phi _{i}(x,y)dx=0
\end{equation*}%
Therefore (\ref{50}) becomes
\begin{equation}
\underset{\varepsilon \rightarrow 0}{\lim }\int_{\Sigma ^{\varepsilon
}}u_{i}^{\varepsilon }(x)\varphi _{i}(x,y)d\sigma ^{\varepsilon
}(x)=\int_{\Omega }\int_{Y_{i}}\nabla _{y}U_{i}(x,y)\phi _{i}(x,y)dydx
\label{51}
\end{equation}%
Again, integrating by parts with respect to $y$\ in the right hand side of (%
\ref{51}) and using (\ref{47}), (\ref{48}) yield
\begin{eqnarray}
\int_{Y_{i}}\nabla _{y}U_{i}(x,y)\phi _{i}(x,y)dydx
&=&-\int_{Y_{i}}U_{i}(x,y)\mathrm{div}_{y}\phi _{i}(x,y)dydx+  \notag \\
&&\left( -1\right) ^{i-1}\int_{\Sigma }U_{i}(x,y)(\phi _{i}(x,y)\cdot \nu
(y))d\sigma (y)dx  \notag \\
&=&\int_{\Sigma }U_{i}(x,y)\varphi _{i}(x,y)d\sigma (y)dx.  \label{52}
\end{eqnarray}%
Finally, by combining (\ref{51}) and (\ref{52}) we arrive at the desired
result. This proves the Lemma.\hfill
\end{proof}

Now we are in a position to give the two-scale homogenized problem.

\begin{proposition}
The two-scale limit $(u_{i},U_{i})\in H_{0}^{1}(\Omega )\times L^{2}(\Omega
,H_{per}^{1}(Y_{i})/\mathbb{R})$ is the solution of the two-scale
homogenized system%
\begin{equation}
\left\{
\begin{array}{l}
-\mathrm{div}_{y}(A_{i}(y)(\nabla u_{i}(x)+\nabla _{y}U_{i}(x,y)))=0\text{
in }\Omega \times Y_{i}\text{,\qquad }i=1,2 \\
\  \\
-\mathrm{div}_{x}(\int_{Y_{i}}A_{i}(\nabla u_{i}+\nabla
_{y}U_{i})dy)+\int_{Y_{i}}a_{i}(y)u_{i}(y)dy+ \\
\  \\
+\left( -1\right) ^{i-1}\int_{\Sigma }\alpha (U_{1}-U_{2}))d\sigma (y)=g_{i}%
\text{ in }\Omega \text{,\qquad }i=1,2 \\
\  \\
(A_{1}(\nabla u_{1}+\nabla _{y}U_{1}))\cdot \nu =(A_{2}(\nabla u_{2}+\nabla
_{y}U_{2}))\cdot \nu \text{ on }\Omega \times \Sigma \text{,} \\
\  \\
(A_{1}(\nabla u_{1}+\nabla _{y}U_{1}))\cdot \nu =-\alpha (u_{1}-u_{2})=0%
\text{ on }\Omega \times \Sigma \text{,} \\
\  \\
u_{i}=0\text{ on }\Gamma \text{,\ }y\mapsto U_{i}(x,y)\text{ is }Y\text{%
-periodic.\qquad }i=1,2%
\end{array}%
\right.  \label{53}
\end{equation}
\end{proposition}

\begin{proof}
Let $\varphi _{i}(x)\in \mathcal{D(}\Omega \mathcal{)},$ $\Phi _{i}\in
\mathcal{D(}\Omega ,\mathcal{C}_{per}^{\infty }(Y_{i})\mathcal{)\ }i=1,2$.
We take $v_{i}(x)=\varphi _{i}(x)+\varepsilon \Phi _{i}(x,\dfrac{x}{%
\varepsilon })$ in the weak formulation (\ref{30}). We have
\begin{eqnarray*}
a^{\varepsilon }((u_{1}^{\varepsilon },u_{2}^{\varepsilon }),(v_{1},v_{2}))
&=&I_{1}^{\varepsilon }+I_{2}^{\varepsilon }+I_{3}^{\varepsilon
}+I_{4}^{\varepsilon }+\varepsilon J_{1}^{\varepsilon }+\varepsilon
J_{2}^{\varepsilon } \\
L^{\varepsilon }((v_{1},v_{2})) &=&K_{1}^{\varepsilon }+K_{2}^{\varepsilon
}+\varepsilon L_{1}^{\varepsilon }+\varepsilon L_{2}^{\varepsilon }
\end{eqnarray*}%
where%
\begin{equation*}
I_{i}^{\varepsilon }=\int_{\Omega _{i}^{\varepsilon }}\left\{ A_{i}(\frac{x}{%
\varepsilon })\nabla u_{i}^{\varepsilon }(x)(\nabla \varphi _{i}(x)+\nabla
_{y}\Phi _{i}(x,\frac{x}{\varepsilon })+a_{i}(\frac{x}{\varepsilon }%
)u_{i}^{\varepsilon }(x)\varphi _{i}(x)\right\} dx,\ i=1,2
\end{equation*}%
\begin{equation*}
I_{3}^{\varepsilon }=\int_{\Sigma ^{\varepsilon }}\varepsilon \alpha (\frac{x%
}{\varepsilon })(u_{1}^{\varepsilon }(x)-u_{2}^{\varepsilon }(x))(\Phi
_{1}(x,\frac{x}{\varepsilon })-\Phi _{2}(x,\frac{x}{\varepsilon }))d\sigma
^{\varepsilon }(x),
\end{equation*}%
\begin{equation*}
J_{i}^{\varepsilon }=\int_{\Omega _{i}^{\varepsilon }}A_{i}(\frac{x}{%
\varepsilon })\nabla u_{i}^{\varepsilon }(x)\nabla _{x}\Phi _{i}(x,\frac{x}{%
\varepsilon })dx,\ i=1,2
\end{equation*}%
\begin{equation*}
I_{4}^{\varepsilon }=\int_{\Sigma ^{\varepsilon }}\alpha (\frac{x}{%
\varepsilon })(u_{1}^{\varepsilon }(x)-u_{2}^{\varepsilon }(x))(\varphi
_{1}(x)-\varphi _{2}(x))d\sigma ^{\varepsilon }(x);
\end{equation*}%
and
\begin{equation*}
K_{i}^{\varepsilon }=\int_{\Omega _{i}^{\varepsilon }}f_{i}^{\varepsilon
}(x)\varphi _{i}(x)dx;\ L_{i}^{\varepsilon }=\int_{\Omega _{i}^{\varepsilon
}}f_{i}^{\varepsilon }(x)\Phi _{i}(x,\frac{x}{\varepsilon })dx,\ i=1,2.
\end{equation*}%
Letting $\varepsilon \rightarrow 0$ and using (\ref{42})-(\ref{45}), we have%
\begin{equation}
\left.
\begin{array}{c}
\lim I_{i}^{\varepsilon }=\int_{\Omega \times Y_{i}}\left\{ A_{i}(y)(\nabla
u_{i}(x)+\nabla _{y}U_{i}(x,y))(\nabla \varphi _{i}(x)+\nabla _{y}\Phi
_{i}(x,y))dxdy\right\} + \\
\int_{\Omega \times Y_{i}}a_{i}(y)u_{i}(x)\varphi _{i}(x)dxdy,\ i=1,2%
\end{array}%
\right.  \label{54}
\end{equation}%
\begin{eqnarray}
\lim I_{3}^{\varepsilon } &=&\int_{\Omega \times \Sigma }\alpha
(y)(u_{1}(x)-u_{2}(x))(\Phi _{1}(x,y)-\Phi _{2}(x,y))dxd\sigma (y)
\label{55} \\
\lim K_{i}^{\varepsilon } &=&\int_{\Omega \times Y_{i}}f_{i}(x,y)\varphi
_{i}(xdxdy;\ i=1,2  \label{56}
\end{eqnarray}%
By virtue of Lemma \ref{lem2}
\begin{equation}
\lim I_{4}^{\varepsilon }=\int_{\Sigma ^{\varepsilon }}\alpha
(y)(U_{1}(x,y)-U_{2}(x,y))(\varphi _{1}(x)-\varphi _{2}(x))dxd\sigma (y);
\label{57}
\end{equation}%
On the other hand we see that
\begin{equation*}
|J_{i}^{\varepsilon }|+|L_{i}^{\varepsilon }|\leq C,\ i=1,2
\end{equation*}%
where $C$ is a positive constant independent of $\varepsilon $. Hence
\begin{equation}
\lim \varepsilon J_{i}^{\varepsilon }=\lim \varepsilon L_{i}^{\varepsilon
}=0,\ i=1,2  \label{59}
\end{equation}%
Now passing to the limit in (\ref{30}) and using (\ref{54})-(\ref{59}), we
obtain the two-scale system%
\begin{gather}
\sum_{i=1,2}\int_{\Omega \times Y_{i}}\left\{ A_{i}(\nabla u_{i}+\nabla
_{y}U_{i})(\nabla \varphi _{i}+\nabla _{y}\Phi _{i})+a_{i}u_{i}\varphi
_{i}\right\} dxdy+  \notag \\
\   \notag \\
+\int_{\Omega \times \Sigma }\left\{ \alpha (U_{1}-U_{2})(\varphi
_{1}-\varphi _{2})+\alpha (u_{1}-u_{2})(\Phi _{1}-\Phi _{2})\right\}
dxd\sigma (y)  \notag \\
\   \notag \\
=\sum_{i=1,2}\int_{\Omega \times Y_{i}}f_{i}\varphi _{i}dxdy  \label{60}
\end{gather}%
Now by density the formulation (\ref{60}) remains true for all $%
(v_{1},v_{2},V_{1},V_{2})$ in $(H_{0}^{1}(\Omega ))^{2}\times L^{2}(\Omega
,H_{per}^{1}(Y_{i})/\mathbb{R})\times L^{2}(\Omega ,H_{per}^{1}(Y_{2})/%
\mathbb{R})$. An integration by parts yields the two-scale homogenized
system (\ref{53}).
\end{proof}

\begin{proof}[Proof of Theorem \protect\ref{mr}]
First we take $\Phi _{1},\varphi _{2},\Phi _{2}\equiv 0$ in (\ref{60}). Then
integration by parts yields%
\begin{equation}
\left\{
\begin{array}{c}
-\mathrm{div}_{x}\left[ \int_{Y_{1}}A_{1}(y)(\nabla u_{1}(x)+\nabla
_{y}U_{1}(x,y)dy\right] +\int_{Y_{1}}a_{1}(y)u_{1}(x)dy \\
+\int_{\Sigma }\alpha (y)(U_{1}(x,y)-U_{2}(x,y)d\sigma
(y)=\int_{Y_{1}}f_{1}(x,y)dy\text{ in }\Omega , \\
u_{1}=0\text{ on }\Gamma .%
\end{array}%
\right.  \label{61}
\end{equation}%
Similarly we take $\varphi _{1},\Phi _{1},\Phi _{2}$ equal to zero and this
gives%
\begin{equation}
\left\{
\begin{array}{c}
-\mathrm{div}_{x}\left[ \int_{Y_{1}}A_{1}(y)(\nabla u_{1}(x)+\nabla
_{y}U_{1}(x,y)dy\right] +\int_{Y_{1}}a_{1}(y)u_{1}(x)dy \\
+\int_{\Sigma }\alpha (y)(U_{1}(x,y)-U_{2}(x,y)d\sigma
(y)=\int_{Y_{1}}f_{1}(x,y)dy\text{ in }\Omega , \\
\  \\
u_{1}=0\text{ on }\Gamma .%
\end{array}%
\right.  \label{62}
\end{equation}%
Now choosing $\varphi _{1},\varphi _{2},\Phi _{2}=0$ (resp. $\varphi
_{1},\varphi _{2},\Phi _{1}=0$) in (\ref{60}) gives after integration by
parts
\begin{equation}
\left\{
\begin{array}{l}
-\mathrm{div}_{y}(A_{i}(y)(\nabla u_{i}(x)+\nabla _{y}U_{i}(x,y))=0\text{ in
}\Omega \times Y_{i}, \\
\  \\
(A_{1}(y)(\nabla u_{1}(x)+\nabla _{y}U_{1}(x,y)))\cdot \nu _{1}(y) \\
\ \ \ \ \ \ \ \ \ \ \ \ =(A_{2}(y)(\nabla u_{2}(x)+\nabla
_{y}U_{2}(x,y)))\cdot \nu _{1}(y)\text{ on }\Omega \times \Sigma , \\
\  \\
(A_{1}(y)(\nabla u_{1}(x)+\nabla _{y}U_{1}(x,y))\cdot \nu _{1}(y)+\alpha
(y)(u_{1}(x)-u_{2}(x))=0\text{ on }\Omega \times \Sigma , \\
\  \\
y\longmapsto U_{1}(x,y),U_{2}(x,y)\text{ $Y$-periodic.}%
\end{array}%
\right.  \label{h3}
\end{equation}
\end{proof}

Next we shall decouple the problem (\ref{h3}), that is eliminating the
unknowns $U_{1},U_{2}$ from the system (\ref{53}). The linearity of the
problem and the fact that $u_{i}$ do not depend on the fast variable $y$
enable us to put
\begin{equation}
U_{i}(x,y)=\sum_{k=1}^{n}\chi _{i}^{k}(y)\frac{\partial u_{i}}{\partial x_{k}%
}\left( x\right) +\gamma _{i}(y)(u_{1}\left( x\right) -u_{2}\left( x\right)
)+\tilde{u}_{i}\left( x\right) ,\ i=1,2  \label{63}
\end{equation}

Then inserting (\ref{63}) into (\ref{61})and (\ref{62}) yields the equations
(\ref{25})-(\ref{27})

Thus we have proved the Theorem \ref{mr}\hfill

\begin{remark}
It is easy to see that if $A_{i}$ are symmetric for all $i$ then $B_{i}=0$.
\end{remark}

\noindent {\ \textbf{Acknowledgement:\ }This\ work\ is\ part\ of\ the\
Project\ CNEPRU\ N%
${{}^\circ}$%
\ B1602-09-05\ of\ the\ MESRS\ Algerian\ office.\ The\ author\ is\ grateful\
to\ the\ high\ education\ Algerian\ office\ for\ the\ financial\ support.\ }

\end{document}